\documentclass[10pt,a4paper]{article} 
\usepackage{amsfonts}
\usepackage{amsmath}
\usepackage{amssymb}
\usepackage{enumerate}
\usepackage{amsthm} 
\usepackage{ifthen}
\usepackage[final]{graphicx} 
\usepackage{time} 
\usepackage{xspace,url}

\numberwithin{figure}{section}
\numberwithin{equation}{section}

\newcommand{\RootPath}{.}

\newcommand{\ExternalFiguresPath}{\RootPath}

\newcommand{\eop}{\hspace*{\fill}~$\square$} 

\newtheorem{theorem}{Theorem}[section]
\newtheorem{proposition}[theorem]{Proposition}
\newtheorem{lemma}[theorem]{Lemma}

\newtheorem{Conjecture}[theorem]{Conjecture}

\newtheorem*{theorem*}{Theorem}
\newtheorem*{proposition*}{Proposition}
\newtheorem*{lemma*}{Lemma}
\newtheorem*{corollary*}{Corollary}
\newtheorem*{Conjecture*}{Conjecture}
\newenvironment{conjecture*}{\begin{Conjecture*}}{\eop\end{Conjecture*}}

\newtheorem{Problem}[theorem]{Problem}

\newtheorem{Problems}[theorem]{Problems}

\newtheorem*{Problem*}{Problem}
\newenvironment{problem*}{\begin{Problem*}}{\eop\end{Problem*}}
\newtheorem*{Problems*}{Problems}
\newenvironment{problems*}{\begin{Problems*} {\ } \hspace{-.5cm} \begin{enumerate}}{\end{enumerate}\eop\end{Problems*}}
\newtheorem*{Exercise*}{Exercise}
\newenvironment{exercise*}{\begin{Exercise*}}{\eop\end{Exercise*}}

\newcommand{\mycaption}[1]{\centering{\vspace{\medskipamount}\refstepcounter{figure}\textbf{Figure~\thefigure.} {#1}}}

\newcommand{\natur}{\ensuremath{\mathbb{N}}}
\newcommand{\real}{\ensuremath{\mathbb{R}}}

\newcommand{\integer}{\ensuremath{\mathbb{Z}}}

\newcommand{\sprod}[2]{\left<#1,#2\right>}

\newcommand{\setcond}[2]{\left\{ #1 : #2 \right\}}

\newcommand{\MxColTwo}[2]{
    \begin{bmatrix}
        | & |  \\
      {#1} & {#2}  \\
       | & | 
    \end{bmatrix}
}

\newenvironment{FigTab}[2]{
	\begin{figure}[htb]
	\setlength{\unitlength}{#2}
	\begin{center}
	\begin{tabular}{#1}
}{
    \end{tabular}
    \end{center}
    \end{figure}
}

\newcommand{\ExternalFigure}[2]{
        \includegraphics[width=#1\textwidth]{\ExternalFiguresPath/{#2}}
}
\newcommand{\IncludeGraph}[2]{
	\includegraphics[#1]{\ExternalFiguresPath/{#2}}
}

\newcommand{\range}[2]{{#1},\ldots,{#2}}

\newcommand{\conv}{\mathop{\mathrm{conv}}\nolimits}
\newcommand{\lin}{\mathop{\mathrm{lin}}\nolimits}
\newcommand{\aff}{\mathop{\mathrm{aff}}\nolimits}

\newcommand{\intr}{\mathop{\mathrm{int}}\nolimits}

\newcommand{\bd}{\mathop{\mathrm{bd}}\nolimits}

\newcommand{\cl}{\mathop{\mathrm{cl}}\nolimits}

\newcommand{\E}{\mathop{\mathbb{E}}\nolimits}

\newcommand{\sign}{\mathop{\mathrm{sign}}\nolimits}

\begin{document}

\newcommand{\skprj}[2]{\Pi_{#1}^{#2}} 
\newcommand{\Wlg}{Without loss of generality\xspace} 
\newcommand{\wlg}{without loss of generality\xspace}
\renewcommand{\iff}{if and only if\xspace}
\newenvironment{InlineMx}{\left[\begin{smallmatrix}}{\end{smallmatrix}\right]}
\newenvironment{Mx}{\begin{bmatrix}}{\end{bmatrix}}
\newcommand{\IdMx}{E}
\newcommand{\ZMx}{O}
\newcommand{\Rot}{\mathop{\mathcal R}\nolimits}
\newcommand{\RMx}{\Rot}
\newcommand{\MxL}{\left[}
\newcommand{\MxR}{\right]}
\newcommand{\MxSeq}[3]{\MxL{#1}\MxR_{#2}^{#3}} 
\newcommand{\IDKO}{\intr \supp {g_K} \setminus \{o\}} 
\newcommand{\supp}{\mathop{\mathrm{supp}}\nolimits}

\newcommand{\Bcal}{\mathop{\mathcal B}\nolimits}
\newcommand{\hdmeas}{\mathcal{H}}
\newcommand{\eps}{\varepsilon}
\newcommand{\usphere}{\mathop{\mathbb{S}}\nolimits}
\newcommand{\uball}{\mathop{\mathbb{B}}\nolimits}
\newcommand{\XX}{\mathop{\mathcal X}\nolimits}
\newcommand{\Ccal}{\mathop{\mathcal C}\nolimits}
\newcommand{\cclass}{\mathop{\mathrm C}\nolimits} 
\newcommand{\bdclass}{\mathcal{C}} 
\newcommand{\Ucal}{\mathop{\mathcal U}\nolimits}
\newcommand{\dd}{\mathrm{d}}
\newcommand{\dotvar}{\,.\,}
\newcommand{\transp}{\top}
\newcommand{\hess}{\nabla^2}
\newcommand{\impl}[2]{\eqref{#1} $\Longrightarrow$ \eqref{#2}}
\newcommand{\ThmSource}[1]{\emph{(#1).}}
\newcommand{\ThmTitle}[2][]{\ifthenelse{\equal{#1}{}}{\emph{(#2)}}{\emph{(#2; #1)}}}
\newcommand{\notion}[2][]{\emph{#2}\xspace} 
\newcommand{\dist}{\mathop{\mathrm{dist}}}

\newcommand{\mycite}[2]{\ifthenelse{\equal{#2}{}}{\cite{#1}}{\cite[#2]{#1}}\xspace}
\newcommand{\MatheronReprint}[1][]{\mycite{Matheron8601}{#1}}
\newcommand{\MatRepr}[1][]{\mycite{Matheron8601}{#1}} 
\newcommand{\SchnBk}[1][]{\mycite{MR94d:52007}{#1}} 
\newcommand{\OrMatrBk}[1][]{\mycite{MR1744046}{#1}} 
\newcommand{\CvBstKnown}[1][]{\mycite{MR2108257}{#1}} 
\newcommand{\GardZhang}[1][]{\mycite{MR1623396}{#1}}
\newcommand{\BMIneqPaper}[1][]{\mycite{MR1898210}{#1}}
\newcommand{\CovBstKnown}[1][]{\mycite{MR2108257}{#1}}
\newcommand{\Nagel}[1][]{\mycite{MR1232748}{#1}}
\newcommand{\VolIntersect}[1][]{\mycite{MR1260892}{#1}}

\newcommand{\G}{G}
\newcommand{\hessg}{\det G}
\newcommand{\detu}[2]{\det(u_{#1},u_{#2})}
\newcommand{\detuvar}[3]{\det(u_{#1}(#3),u_{#2}(#3))}

\title{Confirmation of Matheron's Conjecture\\ on the Covariogram of  a Planar Convex Body}
\date{\small \today}
\author{\small Gennadiy Averkov and Gabriele Bianchi}
\maketitle

\begin{abstract}
The covariogram $g_K$ of a convex body $K$ in $\E^d$ is the function
which associates to each $x \in \E^d$ the volume of the intersection of
$K$ with $K+x$. In 1986 G.~Matheron conjectured that for $d=2$ the covariogram
$g_K$ determines $K$ within the class of all planar convex
bodies, up to translations and reflections in a point.
This problem is equivalent to some problems in stochastic geometry and
probability as well as to a particular case of the phase retrieval problem
in Fourier analysis. It is also relevant for the inverse
problem of determining the atomic structure of a quasicrystal from its
X-ray diffraction image.
In this paper we confirm Matheron's conjecture completely.

\newtheoremstyle{itsemicolon}{}{}{\mdseries\rmfamily}{}{\itshape}{.}{ }{}
\theoremstyle{itsemicolon}
\newtheorem*{msc*}{2000 Mathematics Subject Classification}

\begin{msc*}
	Primary 60D05; Secondary 52A10, 52A22, 52A38, 42B10
\end{msc*}

\newtheorem*{keywords*}{Key words and phrases}

\begin{keywords*}
	Autocorrelation, covariogram,  cut-and-project scheme, geometric tomography, image analysis, phase retrieval, quasicrystal, set covariance
\end{keywords*}
\end{abstract}

\section{Introduction}

Let $C$ be a compact set in the Euclidean space $\E^d, \ d \ge 2.$ The \notion{covariogram} $g_C$ of  $C$ is the function on $\E^d$ defined by
\begin{equation}
	g_C(x) := V_d(C \cap (C+x)), \qquad x \in \E^d, \label{cov:def:eq}
\end{equation}
where $V_d$ stands for the $d$-dimensional Lebesgue measure. This function, which was introduced by G.~Matheron in his book
\cite[Section~4.3]{MR0385969} on random sets, is also called \notion{set
covariance}. The covariogram $g_C$ coincides with the  \notion{autocorrelation} of the characteristic function $\mathbf{1}_C$ of $C,$ i.e.:
\begin{equation}\label{convoluzione}
g_C =\mathbf{1}_C\ast \mathbf{1}_{(-C)}.
\end{equation}
The covariogram $g_C$ is clearly unchanged with respect to  translations and reflections 
of $C$, where, throughout the paper,  \emph{reflection} means reflection 
in a point. A \notion{convex body} in $\E^d$ is a convex compact set with nonempty interior. In 1986 Matheron~\cite[p.~20]{Matheron8601}  asked the following question and conjectured a positive answer for the case $d=2$.

\newtheorem*{covproblem*}{Covariogram Problem}
\begin{covproblem*}
Does the covariogram determine a convex body in $\E^d,$ among all convex bodies, up to translations and reflections?
\end{covproblem*}

We are able to confirm Matheron's conjecture completely. 
\begin{theorem} \label{MatConjConf}
	Every planar convex body is determined within all planar convex bodies by its covariogram, up to translations and reflections.
\end{theorem}

It is known that the covariogram problem is equivalent to any of the following problems.
\begin{enumerate}[{P}1.]
\item \label{chord_lengths} Determine a convex  body $K$ by the knowledge, for each unit vector $u$ in $\E^d$, of the distribution of the lengths of the chords of $K$ parallel to $u$. 
\item \label{distribution_XY} Determine a convex  body $K$ by the  distribution of $X-Y$, where $X$ and $Y$ are independent random variables uniformly distributed over $K$.
\item \label{phase_retrieval} Determine  the characteristic function $\mathbf{1}_K$ of a convex body $K$ from  the modulus of its Fourier transform $\widehat{\mathbf{1}_K}$.
\end{enumerate}

In view of Theorem~\ref{MatConjConf}, for Problems P\ref{chord_lengths} and P\ref{distribution_XY} the determination holds within the class of planar convex bodies and for Problem P\ref{phase_retrieval} within the class of characteristic functions of planar convex bodies. In each of the three problems the determination is unique up to translations and reflections of the body.

The equivalence of the covariogram problem and P\ref{chord_lengths} was observed by Matheron, who showed in \cite[p.~86]{Matheron8601} that the derivatives $\frac{\dd}{\dd r} g_K(ru),$ for all $r>0$, yield the
distribution of the lengths of the chords of $K$ parallel to $u$.
Blaschke \cite[\S4.2]{MR2162874} asked whether the distribution of
the lengths of \emph{all} chords (that is, not separated direction by direction) of a planar convex body determines that body, up to isometries in $\E^2.$ Mallows and Clark~\cite{MR0259976}
constructed polygonal examples that show that the answer is negative in general. Gardner, Gronchi, and Zong  \cite{MR2160045} observed that the distribution of
the lengths of the chords of $K$ parallel to $u$ coincides, up to a
multiplicative factor, with the \notion{rearrangement} of the X-ray of $K$ in
direction $u$, and rephrased P\ref{chord_lengths} in these terms. Chord-length distributions  are of wide interest beyond mathematics, as Mazzolo, Roesslinger, and Gille \cite{MR2023576} describe. See also
Schneider \cite{MR1196706} and Cabo and Baddeley \cite{MR1974301}. 

Problem P\ref{distribution_XY} was asked by Adler and Pyke~\cite{AP91} in 1991; see also \cite{MR1450931}. Its equivalence to the covariogram problem comes from the observation that the convolution in \eqref{convoluzione} is the probability density of $X-Y$, up to a multiplicative factor.

Problem P\ref{phase_retrieval} is a special case of the \emph{phase retrieval problem}, where $\mathbf{1}_K$ is replaced by a function with compact support. The phase retrieval problem has applications in \emph{X-ray crystallography, optics, electron microscopy} and other areas, references to which may be found in \cite{MR1938112}. 
The equivalence of the covariogram problem and P\ref{phase_retrieval} follows by applying the Fourier transform to \eqref{convoluzione} and using the
relation $\widehat{\mathbf{1}_{(-K)}}=\overline{\widehat{\mathbf{1}_K}}$.

Recently, Baake and Grimm \cite{BaakeGrimm} have observed that the covariogram problem is relevant for the inverse problem of finding the atomic structure of a \notion{quasicrystal} from  its \notion[X-ray diffraction image]{X-ray diffraction image}.
It turns out that quasicrystals can often be described by means of the so-called \notion[cut-and-project scheme]{cut-and-project scheme}; see \cite{MR2084582}. In this scheme a quasiperiodic discrete subset $S$ of $\E^d,$ which models the atomic structure of a quasicrystal, is described as the canonical projection of $Z \cap (\E^d \times W)$ onto $\E^d,$ where $W$ (which is called \emph{window})  is a subset of $\E^n, \ n \in \natur,$ and $Z$ is a lattice in $\E^d \times \E^n.$  For many quasicrystals, the lattice $Z$ can be recovered from the diffraction image of $S.$ Thus, in order to determine $S,$ it is necessary to know $W$. The covariogram problem enters at this point, since $g_W$ can be obtained from the diffraction image of $S.$ Note that the set $W$ is in many cases a convex body.

In \cite[Theorem~6.2 and Question~6.3]{MR1623396}  the covariogram problem was transformed to a question for the so-called \notion{radial mean bodies}. 

A planar convex body $K$ can be determined by its covariogram in a 
class $\mathcal C$ of  sets which is much larger than that of
convex bodies.  This is a consequence of Theorem~\ref{MatConjConf} and of a result of
Benassi, Bianchi, and D'Ercole \cite{BenassiBianchiDErcole}. In  \cite{BenassiBianchiDErcole} the class $\mathcal C$ is defined  and it is proved that a body $C\in\mathcal C$ whose
covariogram is equal to that of a convex body is necessarily convex. 
However, in Theorem~\ref{MatConjConf} the assumption that $K$ is convex is crucial, since there exist examples of non-convex sets which are neither translations nor reflections of each other and have equal covariograms; see~\cite{MR2160045}, Rataj \cite{MR2040232}, and \cite{BenassiBianchiDErcole}. 

The first partial solution of Matheron's  conjecture was given by Nagel \cite{MR1232748} in  1993, who confirmed it for all convex polygons. 
% Schmitt \cite{MR1196706}, in the same year, confirmed the conjecture for a class of planar polygons, which includes some non-convex ones, but does not contain convex polygons with parallel edges.  
Schmitt \cite{MR1196706}, in the same year, gave a constructive proof of the determination of each set in a suitable class of polygons by its covariogram. This class  contains each convex polygon without parallel edges and also some non-convex polygons.
In 2002 Bianchi, Segala and Vol\v{c}i\v{c} \cite{MR1938112} gave a positive answer to the covariogram problem for all planar convex bodies whose boundary has strictly positive continuous curvature.
Bianchi \cite{MR2108257} proved a common generalization of this and Nagel's result.
In  \cite{AverkovBianchi0306} the authors of this paper studied how much of the covariogram data is needed for the uniqueness of the determination, and also extended the class of bodies for which the conjecture was confirmed.

The covariogram problem in the general setting has negative answer, as Bianchi \cite{MR2108257} proved by finding counterexamples in $\E^d$ for every $d\geq 4$. For other results in dimensions higher that two we refer to Goodey, Schneider, and Weil \cite[p.~87]{MR1416411}, and \cite{Bianchi0602}.
In \cite{Bianchi0602} it is proved that a convex three-dimensional polytope is determined by its covariogram. This proof requires the following generalization of the covariogram problem. The \emph{cross covariogram} of two convex bodies $K$ and $L$
in $\E^2$ is the function defined for each $x\in\E^2$ by
$g_{K,L}(x):=V_2(K\cap(L+x)).$ Bianchi \cite{Bianchi0602} proves that if $K$ and
$L$ are convex polygons, then $g_{K,L}$ determines \emph{both $K$ and $L$ }, with exclusion of a completely described family of exceptions. The family of exceptions is
composed of pairs of parallelograms.

In view of results from \cite{MR2108257}, for proving Theorem~\ref{MatConjConf} it suffices to derive the following statement. 

\begin{proposition} \label{ArcDetermProp}
	Let $K$ and $L$ be planar strictly convex and $\cclass^1$ regular bodies with equal covariograms. Then $L$ possesses a non-degenerate boundary arc whose translation or reflection lies in the boundary of $K.$ 
\end{proposition}

Theorem~\ref{MatConjConf} follows directly from Proposition~\ref{ArcDetermProp} and the following two statements. 

\begin{theorem} \emph{(Bianchi \cite{MR2108257})} \label{IrregThm} 
	Let $K$ and $L$ be planar convex bodies with equal covariograms. Assume that one of them is not strictly convex or not $\cclass^1$ regular. Then $K$ and $L$ are translations or reflections of each other. 
\end{theorem}

\begin{proposition} \emph{(Bianchi \cite{MR2108257})} \label{FromArcsToBodies}
	Let $K$ and $L$ be planar convex bodies with equal covariograms and a common non-degenerate boundary arc. Then $K$ and $L$ coincide, up to translations and reflections.
\end{proposition}

The ``heart'' of the proof is contained in Section~\ref{sect:rel:pos:par}, which also contains, in the beginning, an explanation of the main ideas.
Many natural questions for the covariogram problem are still open. We mention
here some of them.
\begin{enumerate}
 \item Which four-dimensional convex polytopes are determined by their
covariogram?
 \item All known examples of convex bodies that are not determined by their
covariogram are Cartesian products. Do there exist other examples?
 \item Is the answer to the covariogram problem positive for all three-dimensional convex bodies whose boundary has continuous  and strictly positive principal curvatures?
\end{enumerate}

\section{Preliminaries}

The closure, boundary, interior, linear hull, affine hull, and convex hull of a set, and the support of a function, are abbreviated, in the standard way, by $\cl,$ $\bd,$ $\intr,$  $\lin,$ $\aff,$ $\conv,$ and $\supp,$ respectively.  We denote by $o,$ $\sprod{\dotvar}{\dotvar},$ $|\dotvar|,$ and $\usphere^{d-1},$ origin, scalar product,  Euclidean norm, and Euclidean unit sphere in $\E^d,$ respectively. In analytic expressions elements of $\E^d$ are identified with real column vectors of length $d.$ Thus, $\sprod{x}{y} = x^\transp y,$ where $(\dotvar)^\transp$ denotes matrix transposition. Throughout the paper we use the matrix $\RMx:=\begin{InlineMx} 0 & -1 \\ 1 & \phantom{-}0\end{InlineMx}$ of 90-degree rotation in the counterclockwise orientation. We do not distinguish between $2 \times 2$ matrices over $\real$ and linear operators in $\E^2.$ For vectors $x, y \in \E^2$ we put 
\begin{equation}
	\det(x,y) :=  \det \MxColTwo{x}{y}  = -x^\transp \RMx y, \label{det:def:eq}
\end{equation}
where {\scriptsize $\MxColTwo{x}{y}$} stands for the matrix whose first column is $x$ and the second one is $y.$

Regarding standard notations and notions from the theory of convex sets we mostly follow the monograph \cite{MR94d:52007}. The \notion{difference body} of a convex body $K$ is the set $DK:=K+(-K)=\setcond{x -y}{x, y \in K}.$  It is not hard to see that $\supp g_K = DK.$ A boundary point $p$ of a convex body $K$ is said to be $\cclass^1$ \notion{regular} if there exists precisely one hyperplane supporting $K$ at $p.$ Furthermore, a convex body $K$ is said to be $\cclass^1$ \notion{regular} if all boundary points of $K$ are $\cclass^1$ regular. We say that a convex polygon $P \subseteq \E^2$ is \notion{inscribed} in a convex body $K \subseteq \E^2$ if  all vertices of $P$ lie in $\bd K.$  Given $q_1$, $q_2\in \bd K$, the chord $[q_1,q_2]$ is said to be an \notion{affine diameter} of $K$, if for some $u \in \E^2\setminus\{o\}$ the vectors $u$ and $-u$ are outward normals of $K$ at $q_1$ and $q_2,$ respectively. It is well known that $[q_1,q_2]$ is an affine diameter of $K$  \iff $q_1 - q_2 \in \bd DK.$  If $K$ is a planar convex body and $p, \, q$ are two distinct boundary points of $K,$ then $[p,q]_{K}$ stands for the counterclockwise boundary arc of $K$ starting at $p$ and terminating at $q.$

\section{Gradient of covariogram and inscribed parallelograms} \label{frst:drv:cov}

Let $K$ be a strictly convex and  $\cclass^1$ regular convex body in $\E^2.$ Consider an arbitrary $x \in \IDKO.$ Then there exist points $p_i(K,x), \ i \in \{1,\ldots,4\},$ in counterclockwise order on $K,$ such that $x=p_1(K,x)-p_2(K,x)=p_4(K,x)-p_3(K,x);$ see Fig.~\ref{P:def:fig:1}, also regarding notations introduced below. Then the set 
\begin{equation}
	P(K,x):=\conv \{p_1(K,x),\ldots,p_4(K,x)\} \label{P:def:eq}
\end{equation}
is a parallelogram inscribed in $K,$  whose edges are translates of $[o,x]$ and $[o,D(K,x)]$ with 
$$
	D(K,x):=p_1(K,x)-p_4(K,x),
$$
By $u_i(K,x)$ we denote the outward unit normal of $K$ at $p_i(K,x).$ 

    \begin{FigTab}{cc}{1mm}
        \begin{picture}(54,50)
        \put(-1,-0.7){\IncludeGraph{width=50\unitlength}{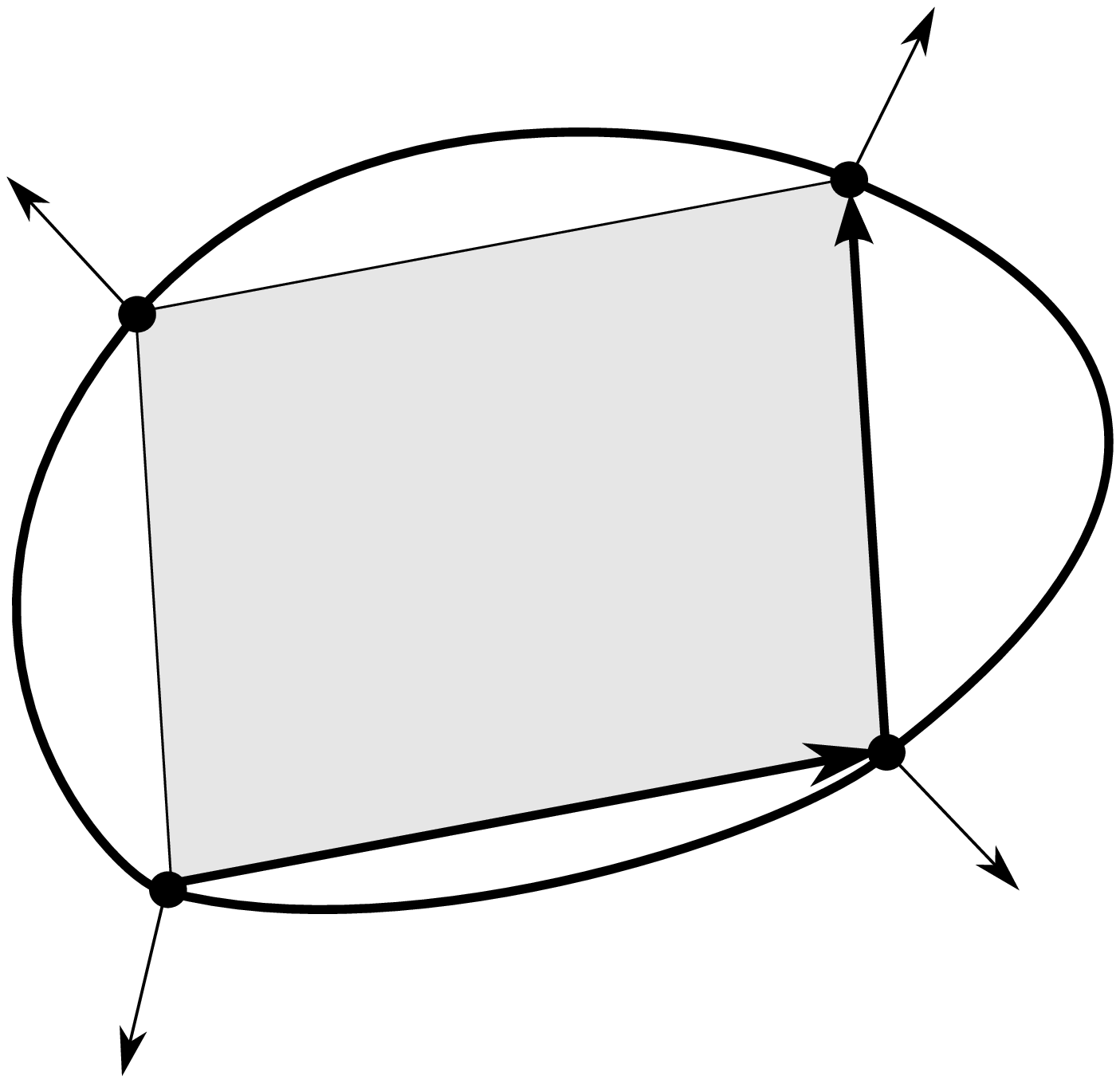}}
        \put(51,29){\small $K$}
        \put(36,14){\small $p_4$}
        \put(34,35){\small $p_1$}
        \put(9,9){\small $p_3$}
        \put(7,30){\small $p_2$}
        \put(46,5){\small $u_4$}
        \put(43,45){\small $u_1$}
        \put(-1,40){\small $u_2$}
        \put(7,-3){\small $u_3$}
	\put(22,10.5){\small $x$}
	\put(31.7,24){\small $D(x)$}

        \end{picture}
    \\
    \parbox[t]{0.90\textwidth}{\mycaption{\label{P:def:fig:1}}}
    \end{FigTab}

For the sake of brevity, dealing with functionals $f(K,x)$ depending on $K$ and $x,$  we shall also use  the notations $f(x)$ or $f$ instead of $f(K,x),$ provided the choice of $K$ and/or $x$ is clear from the context.

\begin{theorem} \label{frst:drv:cov:thm} 
	Let $K$ be a strictly convex and $\cclass^1$ regular body in $\E^2$ and let $x \in \IDKO.$ Then the following statements hold. 
	\begin{enumerate}[I.]
		\item  \label{DReprStat} The covariogram $g_K$ is continuously differentiable at $x.$ Moreover, 
		\begin{equation}
			\nabla g_K(x) = \RMx (D(x)). \label{D:repr}
		\end{equation}
		\item \label{PDContCond} The functions  $P, \ D, \ p_i, \ u_i$ with $ i \in \{1,\ldots,4\}$ are continuous at $x.$ 
		\item \label{Ptrans:prop} For every strictly convex $\cclass^1$ regular body $L$ with $g_L=g_K,$  the parallelogram $P(L,x)$ is a translate of $P(K,x).$ 
	\end{enumerate}
\end{theorem}
\begin{proof}
	Part~\ref{DReprStat} is known; see \MatRepr[p.~3] and \VolIntersect[p.~282].  Let us prove Part~\ref{PDContCond}.   If $x \in \IDKO,$ then $\bd K$ and $\bd K +x$ intersect precisely at $p_1(x)$ and $p_4(x).$ Furthermore, the intersection is transversal. By the Implicit Function Theorem, this implies that $p_1(x)$ and $p_4(x)$ depend continuously on $x.$ Since $K$ is $\cclass^1$ regular, the outward unit normal $u(p)$ of $K$ at a boundary point $p$ of $K$ depends continuously on $p.$ Therefore, for $i \in \{1,\ldots,4\}$ the function $u_i(x)=u(p_i(x))$ depends continuously on $x.$  Part~\ref{Ptrans:prop} follows directly from \eqref{D:repr}. 
\end{proof}

In what follows, notations involving an integer subscript $i$ ranging in a certain interval are extended periodically to all $i \in \integer.$ For example, we set $p_i(K,x):=p_j(K,x),$ where $i \in \integer,$ $j \in \{1,\ldots,4\}$ and $i=j\,(\!\!\!\!\mod 4).$

Throughout the paper, the parallelograms $P(K,x)$ will provide a convenient geometric representation of some information contained in the covariogram, in view of \eqref{D:repr}. A priori, for two planar convex bodies $K$ and $L$ with $g_K=g_L$, the translation
that carries $P(K,x)$ to $P(L,x)$ may depend on $x.$ One crucial step of the proof is to show that the above translation is in fact independent of $x.$ See  the beginning of Section~\ref{sect:rel:pos:par} for a brief sketch of the mentioned argument.

\section{Second derivatives, Monge-Amp\`{e}re equation, and  central \\ symmetry} \label{scnd:drv:sect}

If the covariogram of $K$ is twice differentiable at $x,$ we introduce the Hessian matrix 
$$
	\G(K,x) := \left[\frac{\partial^2 g_K(x)}{\partial x_i \partial x_j}\right]_{i,j=1}^2.
$$
The relations given in Theorem~\ref{scnd:drv:cov:thm} are reformulations of the relations presented in \cite[pp.~12-18]{Matheron8601}. Part~\ref{GRel} of Theorem~\ref{scnd:drv:cov:thm} is extended to every dimension in \cite{MR1260892}. We omit the proof of Part~\ref{GRel} and present a short proof of Parts~\ref{DeltaRel} and \ref{uEq}.

\newcommand{\pdp}[1]{\frac{\partial^+}{\partial #1}}
\newcommand{\pdm}[1]{\frac{\partial^-}{\partial #1}}
\newcommand{\pdpm}[1]{\frac{\partial^\pm}{\partial #1}}
\begin{theorem}	\label{scnd:drv:cov:thm} Let $K$ be a strictly convex and $\cclass^1$ regular body in $\E^2.$ Then $g_K(x)$ is continuously differentiable at every $x \in \IDKO.$ Furthermore, for every $x \in \IDKO,$ the following statements hold true.
\begin{enumerate}[I.]
	\item  \label{GRel} The Hessian $\G(x)$ can be represented by 
\begin{equation}
	\G= \frac{u_2 u_1^\transp}{\detu{2}{1}} - \frac{u_3 u_4^\transp}{\detu{3}{4}} = \frac{u_1 u_2^\transp}{\detu{2}{1}}- \frac{u_4u_3^\transp}{\detu{3}{4}}. \label{G:mx:repr} 
\end{equation}
	\item \label{DeltaRel} The determinant of $\G(x)$ depends continuously on $x$ and satisfies 
	\begin{eqnarray}
		\hessg&=& - \frac{\detu{2}{3} \detu{4}{1}}{\detu{3}{4} \detu{1}{2}} < 0, \label{Delta:det:eq:ineq} \\
		1+\hessg&=& \phantom{-}\frac{\detu{2}{4}\detu{1}{3}}{\detu{3}{4}\detu{1}{2}}. \label{OnePlusDelta} 
	\end{eqnarray}
	\item \label{uEq} The vectors $u_1$, $u_3$ and the matrix $\G$ are related by 
	\begin{equation}
		u_1^\transp \G^{-1} u_3 = 0. \label{u1:u3:sym:eq}\\ 
	\end{equation}
\end{enumerate}
\end{theorem}
\begin{proof}
	\emph{Part \ref{GRel}:} For the proof see \cite[pp.~12-18]{Matheron8601} and \cite[pp.~283-284]{MR1260892}.

	\emph{Part \ref{DeltaRel}:} From \eqref{G:mx:repr} we get 
	$$
	\begin{array}{rllll}
		\G \Rot u_2&= &- \frac{u_4u_3^\transp\Rot u_2 }{ \detu{3}{4} } &\stackrel{\eqref{det:def:eq}}{=} & \phantom{-}\frac{\detu{3}{2}}{\detu{3}{4}} u_4, \\
		\G \Rot u_3&= &\phantom{-} \frac{u_1 u_2^\transp \Rot u_3 }{\detu{2}{1}} &\stackrel{\eqref{det:def:eq}}{=} &-\frac{\detu{2}{3}}{\detu{2}{1}} u_1 = \frac{\detu{3}{2}}{\detu{2}{1}} u_1.
	\end{array}
	$$
	The above two equalities imply
	$$
		\G \Rot \MxColTwo{u_2}{u_3} = \detu{3}{2} \MxColTwo{u_4}{u_1} \begin{Mx} \frac{1}{\detu{3}{4}} & 0 \\ 0 & \frac{1}{\detu{2}{1}} \end{Mx}.
	$$
	Taking determinants of the left and the right hand side we obtain 
	\begin{equation} \label{07.11.19,17:17}
		\hessg \cdot \detu{2}{3} = \detu{2}{3}^2 \cdot \detu{4}{1} \cdot \frac{1}{\detu{3}{4} \detu{2}{1}},
	\end{equation}
	Let us notice that 
	\begin{equation} \label{07.10.24,13:31}
		\detu{i}{i+1} >0
	\end{equation}  for every $i \in \{1,\ldots,4\}.$  For instance, $\detu{1}{2}$ is positive because, if $w$ denotes the unit outward normal to the edge $[p_4,p_1]$ of $P(K,x),$ then $w, u_1, u_2,$ and $-w$ are in this counterclockwise order on $\usphere^1,$ by the strict convexity of $K.$ By \eqref{07.10.24,13:31}, we may divide \eqref{07.11.19,17:17} by $\detu{2}{3}$, arriving at the equality in \eqref{Delta:det:eq:ineq}. The inequality in \eqref{Delta:det:eq:ineq} follows from \eqref{07.10.24,13:31}.

	Equality \eqref{OnePlusDelta} follows directly from \eqref{Delta:det:eq:ineq} and 
the algebraic identity
\begin{equation}
	\det(v_1,v_3) \det(v_2,v_4) = \det(v_2,v_3) \det(v_1,v_4) + \det(v_4,v_3) \det(v_2,v_1) \label{gr:plu:2d}
\end{equation}
which holds for all $v_1,\ldots,v_4 \in \E^2$ and can be found, in a much more general form, in \OrMatrBk[p.~127].
 The continuity of $\det G(x)$ is a consequence of \eqref{Delta:det:eq:ineq} and Theorem~\ref{frst:drv:cov:thm}.

	\emph{Part~\ref{uEq}:} We multiply \eqref{G:mx:repr} by $u_1^\transp \Rot$ from the left and by $\Rot u_3$ from the right getting $u_1^\transp \Rot \G \Rot u_3 = 0.$ Expressing the entries of $\Rot \G \Rot $ over the entries of $G$ one can see that $\Rot \G \Rot = -\det \G \cdot G^{-1}.$ Hence, taking into account that $\det \G \ne 0,$ we arrive at \eqref{u1:u3:sym:eq}.
\end{proof}

\begin{theorem}\label{central symmetry}
Let $K$ be a strictly convex and $\cclass^1$ regular body in $\E^2$. The following conditions are equivalent.
\begin{enumerate}[(i)]
 \item\label{cs_one}
The body $K$ is centrally symmetric.
 \item\label{cs_two}
At least one diagonal of each parallelogram inscribed in $K$ is an affine diameter of $K$.
 \item\label{cs_three}
The covariogram $g_K$ is a solution of the \emph{Monge-Amp\`ere differential equation} $\det \G(x) =-1$ for $x \in \IDKO$.
\end{enumerate}
\end{theorem}
\begin{proof}
The implication \eqref{cs_one}$\Longrightarrow$\eqref{cs_two} is trivial. The equivalence of \eqref{cs_two} and \eqref{cs_three} follows from  \eqref{OnePlusDelta}.  It remains to prove that \eqref{cs_two} implies \eqref{cs_one}.

Let us first prove that \eqref{cs_two} implies that both diagonals of each parallelogram inscribed in $K$ are affine diameters. Assume the contrary. Then, for some $x\in \IDKO$,  exactly one diagonal of $P(x)$, say $[p_1(x),p_3(x)]$, is an affine diameter. 
Let $q(t)$, $t\in[0,1]$, be a continuous parametrization of a small arc of $\bd K$ with $q(0)=p_4(x)$. If we define $x(t):=q(t)-p_3(x)$ then $p_3(x(t))=p_3(x)$ and $p_4(x(t))=q(t)$. We claim that there exists a sufficiently small $t>0$ such that no diagonal of $P(x(t))$ is an affine diameter of $K$. In fact, $[p_2(x(t)),p_4(x(t))]$ is not an affine diameter, because it is close to $[p_2(x),p_4(x)],$ which is not an affine diameter. On the other hand, assume that there exists $\eps>0$ such that for each $t\in[0,\eps]$ the diagonal $[p_1(x(t)),p_3(x(t))]$ is an affine diameter. Since $p_1(x)$ is the only point of $\bd K$ with unit outer normal opposite to the one in $p_3(x)$ and $p_3(x)=p_3(x(t))$, we have $[p_1(x(t)),p_3(x(t))]=[p_1(x),p_3(x)]$. Thus,  the reflection  of $\setcond{q(t)}{t\in[0,\eps]}$ about the midpoint of $[p_1(x),p_3(x)]$ is contained in $\bd K$, and this implies that $[p_2(x),p_4(x)]$  is an affine diameter too, a contradiction. This proves the claim.

It remains to prove that if both diagonals of each parallelogram inscribed in $K$ are affine diameters, then $K$ is centrally symmetric. Let $x\in \IDKO$, let $q(t)$, $t\in[0,1]$, be a continuous parametrization of the arc $[p_1(x),p_3(x)]_K$ and let $x(t)$ be as above. Arguing as above one can prove that, for each $t$, we have $[p_1(x(t)),p_3(x(t))]=[p_1(x),p_3(x)]$. Therefore, for each $t$, $q(t)$ and its reflection about the midpoint $c$ of  $[p_1(x),p_3(x)]$ belong to $\bd K$. Thus  $K$ is centrally symmetric with respect to $c$. 
\end{proof}

\section{Determination of  an arc of the boundary} \label{sect:rel:pos:par}

The crucial point of the proof of Proposition~\ref{ArcDetermProp}  is the statement that outer normals of $K$
are determined by $g_K$, up to the ambiguities arising from reflections
of the body. More precisely, we need to prove the following.

\begin{proposition} \label{u1u3:determ}
	Let $K$ be a strictly convex and $\cclass^1$ regular body in $\E^2.$ Then, for every $x \in \IDKO$ with $\det G(x) \ne -1,$ the set $\{u_1(x),-u_3(x)\}$ is uniquely determined by $g_K.$ \eop
\end{proposition}

Let us sketch the proof of Proposition~\ref{u1u3:determ}. First, we prove that there is $y$ such that $P(K,x)$ and $P(K,y)$ have the opposite vertices $p_1$ and $p_3$ in common and $[p_1,p_3]$ is not an affine diameter. This clearly implies that $u_1(K,x) =u_1(K,y), \ u_3(K,x)=u_3(K,y),$ and $u_1(K,x) \ne -u_3(K,x).$ Thus, $u_1$ and $u_3$ satisfy the system given by the two equations obtained by evaluating \eqref{u1:u3:sym:eq}  at both $x$ and $y.$ Using the geometric interpretation of the action of $G$ contained in Lemma~\ref{HessQuadInterp}, in Lemma~\ref{u1u3eigv} we express the vectors $u_1$ and $u_3$ in terms of the eigenvectors of $G(x) G(y)^{-1}.$ In order to make this expression of $u_1$ and $u_3$ dependent only on the covariogram, it remains to prove that the property that $P(K,x)$ and $P(K,y)$ share a diagonal is preserved across bodies with equal covariograms. The latter is done in Proposition~\ref{HexDetermProp}.

Let us now sketch how Proposition~\ref{ArcDetermProp} follows from Proposition~\ref{u1u3:determ}.
Let $K$ and $L$ be strictly convex $\cclass^1$ regular bodies with $g_K=g_L$, and
let us choose $x_0 \in \intr \supp g_K\setminus \{o \}$ such that $\det G(x_0)\neq
-1$. We will prove the following claim:

\textit{ If $x$ belongs to a suitable neighborhood $U$ of $x_0$, and if
$P(K,x)$
and $P(K,x_0)$ share their vertex $p_3$ (i.e. $p_3(K,x)=p_3(K,x_0)$),
then also $P(L,x)$ and $P(L,x_0)$ share their vertex $p_3$.}

Indeed, Proposition~\ref{u1u3:determ} together with some continuity argument 
allows us to prove that when $x$ is close to $x_0$ and
$u_3(K,x)=u_3(K,x_0)$ then we have $u_3(L,x)=u_3(L,x_0)$. In view of the
strict convexity of $K$ and $L$, this implies the claim above.

Let now $x(t)$, for $t\in[0,1]$, be a  parametrization of 
a curve contained in $U$ with the property that, for each $t \in [0,1]$ the parallelograms
$P(K,x_0)$ and $P(K,x(t))$ share their vertex $p_3$. The previous claim
implies that the arc of $\bd K$ spanned by the vertex $p_4(K,x(t))$ when
 $t$ varies in $[0,1]$, is a translate of the  arc of $\bd L$ spanned
by the vertex $p_4(L,x(t))$. Therefore, up to translations, $\bd K$ and
$\bd L$ have an arc in common.

\begin{lemma} \label{HessQuadInterp}
	Let $K$ be a strictly convex and $\cclass^1$ regular body in $\E^2$ and let $x \in \IDKO.$ Let $h \in \E^2 \setminus \{o\}$ be such that the vectors $u_1, u_2, -\frac{h}{|h|}, u_3, u_4, \frac{h}{|h|}$ and  $u_1$ are in counterclockwise order on $\usphere^1.$ 
	 Consider  the convex  quadrilateral $Q(x,h)$ with consecutive vertices $\range{q_1(x,h)}{q_4(x,h)}$ such that $q_1(x,h)=h$, $q_3(x,h)=o$ and, for each $i \in \{1,\ldots,4\},$  the vector $u_i(x)$ is an outward normal of the side $[q_i(x,h),q_{i+1}(x,h)];$  see Fig.~\ref{HessActionFig}. Then   $q_4(x,h)-q_2(x,h)=-\Rot \G(x) h.$
\end{lemma}
	\begin{FigTab}{c}{1mm}
        \begin{picture}(50,42)
	\put(3,2){\IncludeGraph{width=40\unitlength}{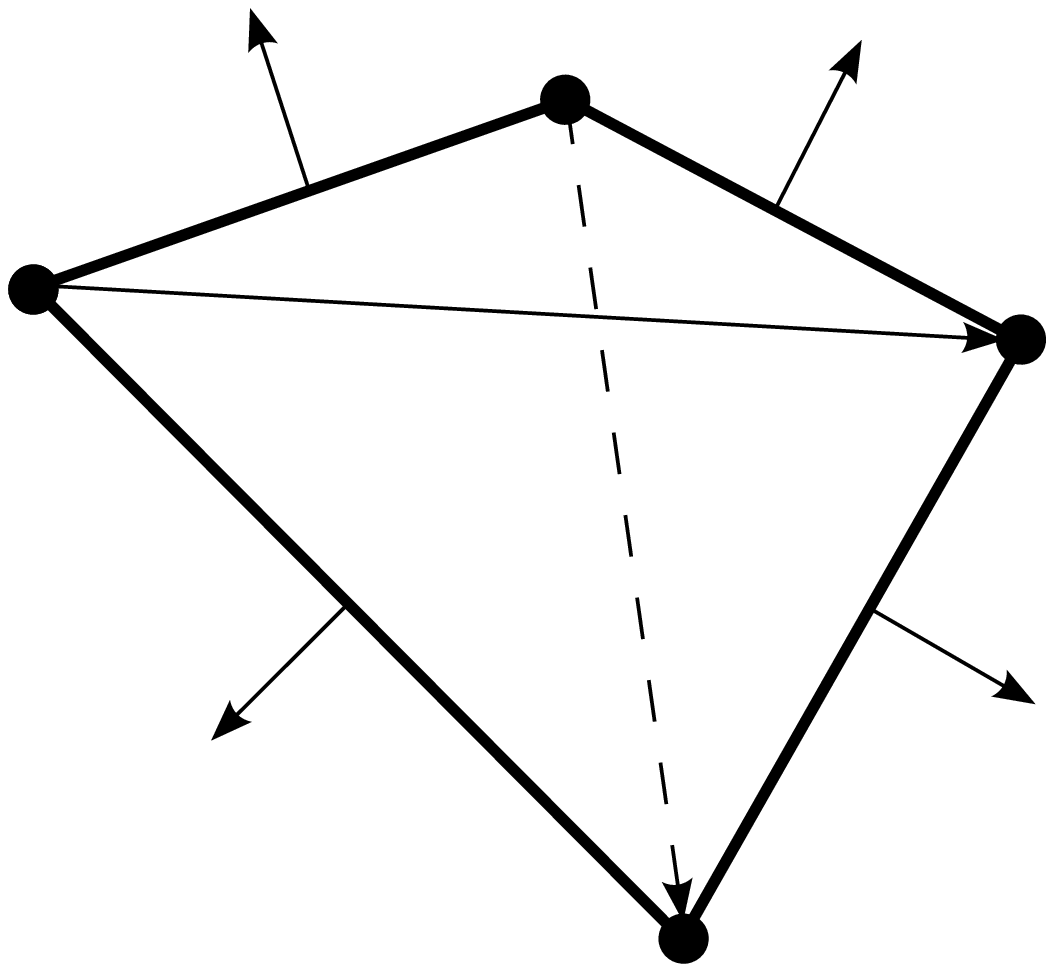}}
	\put(39,38){\small $u_1$}
	\put(9,38){\small $u_2$}
	\put(9,8){\small $u_3$}
	\put(45,10){\small $u_4$}
	\put(0,28){\small $q_3$}
	\put(33,24){\small $h$}
	\put(20,18){\scriptsize $-\Rot \G(x) h$}
	\put(46,26){\scriptsize $q_1$}
	\put(24,38){\scriptsize $q_2$}
	\put(32,1){\scriptsize $q_4$}

        \end{picture}
	\\
	\parbox[t]{0.90\textwidth}{\mycaption{\label{HessActionFig} The quadrilateral $Q(x,h)$.}}
	\end{FigTab}

\begin{proof}
Given two linearly independent vectors $v_1, v_2 \in \E^2,$ we denote by $\skprj{v_1}{v_2}$ the operator of projection onto $\lin v_1$ along the vector $v_2,$ that is,  $\skprj{v_1}{v_2} y := \alpha_1 v_1$ for $y = \alpha_1 v_1 + \alpha_2 v_2$ and $\alpha_1,\alpha_2 \in \real.$ By \notion{Cramer's rule}, $\alpha_1= \det(y,v_2)/\det(v_1,v_2),$ and hence
\begin{equation*}
	\skprj{v_1}{v_2} y = \frac{\det(v_2,y)}{\det(v_2,v_1)} v_1 \stackrel{\eqref{det:def:eq}}{=} - \frac{v_1 v_2^\transp \RMx y}{\det(v_2,v_1)},
\end{equation*}
which implies 
\begin{equation*}
	\skprj{v_1}{v_2}  = - \frac{v_1 v_2^\transp \RMx }{\det(v_2,v_1)}. \label{prj:repr}
\end{equation*}
Thus, \eqref{G:mx:repr} is equivalent to
\begin{equation}\label{expressionG}
 \G= \skprj{u_1}{u_2}\RMx-\skprj{u_4}{u_3}\RMx.
\end{equation}
It can be easily verified that 
$\skprj{\RMx v_1}{\RMx v_2} = - \RMx \skprj{v_1}{v_2} \RMx.$ Therefore, \eqref{expressionG} implies
	\begin{equation} \label{G:R:prj:repr}
		-\RMx\G= \skprj{\RMx u_1}{\RMx u_2}-\skprj{\RMx u_4}{\RMx u_3}. 
	\end{equation}
	We have 
	\begin{align*}
		q_4(x,h) - q_2(x,h) & = \bigl(q_1(x,h) - q_2(x,h)\bigr) + \bigl(q_4(x,h) - q_1(x,h)\bigr)  \\ & = \Pi_{\Rot u_1}^{\Rot u_2} \, h - \Pi_{\Rot u_4}^{\Rot u_3} \, h \stackrel{\eqref{G:R:prj:repr}}{=} -\Rot \, \G \, h.
	\end{align*}
\end{proof}

\begin{lemma} \label{u1u3eigv} Let $K$ be a strictly convex and $\cclass^1$ regular body in $\E^2 \! .$ Let $x, \ y$ be distinct vectors from $\IDKO$ such that
$p_i:=p_i(x)=p_i(y)$ for $i \in \{1,3\}$ and  the  segment $[p_1,p_3]$ is not an affine diameter of $K.$ Then the matrix $\G(x)\G(y)^{-1}$ has two
distinct real eigenvalues. Furthermore, if $v_1, \ v_3 \in \usphere^1$ are distinct eigenvectors of $G(x) G(y)^{-1}$  satisfying $\sprod{x}{v_1} \ge 0, \ \sprod{x}{v_3} \ge 0,$  then 
$\{u_1(x),-u_3(x) \} = \{v_1,v_3 \}.$ 
\end{lemma}
\begin{proof}
The assumptions $p_i(x)=p_i(y)$ for $i \in \{1,3\}$ imply that $u_i:=u_i(x)=u_i(y)$ for $i \in \{1,3\}.$ By \eqref{u1:u3:sym:eq} applied at $x$ and $y,$ we get
	\begin{eqnarray}
		u_1^\transp \G(x)^{-1} u_3 & = & 0, \label{v1v4x:eq} \\
		u_1^\transp \G(y)^{-1} u_3 & = & 0. \label{v1v4y:eq}
	\end{eqnarray}
	From \eqref{v1v4x:eq} and \eqref{v1v4y:eq} we see that $u_1$ is orthogonal to both $\G(x)^{-1} u_3$ and $\G(y)^{-1} u_3.$ Then $\G(x)^{-1} u_3$ and $\G(y)^{-1} u_3$ are parallel, which implies that $\G(x) \G(y)^{-1} u_3$ is parallel to $u_3.$ Thus, $u_3$ is an eigenvector of the matrix $\G(x)\G(y)^{-1}.$ Analogous arguments show that also $u_1$ is an eigenvector of $\G(x) \G(y)^{-1}.$ 
	We show that it is not possible that all vectors from $\E^2 \setminus \{o\}$ are eigenvectors of  $\G(x) \G(y)^{-1}.$  We introduce the centrally symmetric hexagon $H:=\conv (P(x) \cup P(y)).$ After, possibly, interchanging the roles of $x$ and $y,$ we assume that the points $p_1(x), \ p_2(y), \ p_2(x), \ p_3(x)$ are in counterclockwise order on $\bd K$; see Fig.~\ref{HexNormDetermFig}. 

	\begin{FigTab}{cc}{1mm}
	\begin{picture}(45,60)
	\put(-1,5){\IncludeGraph{width=45\unitlength}{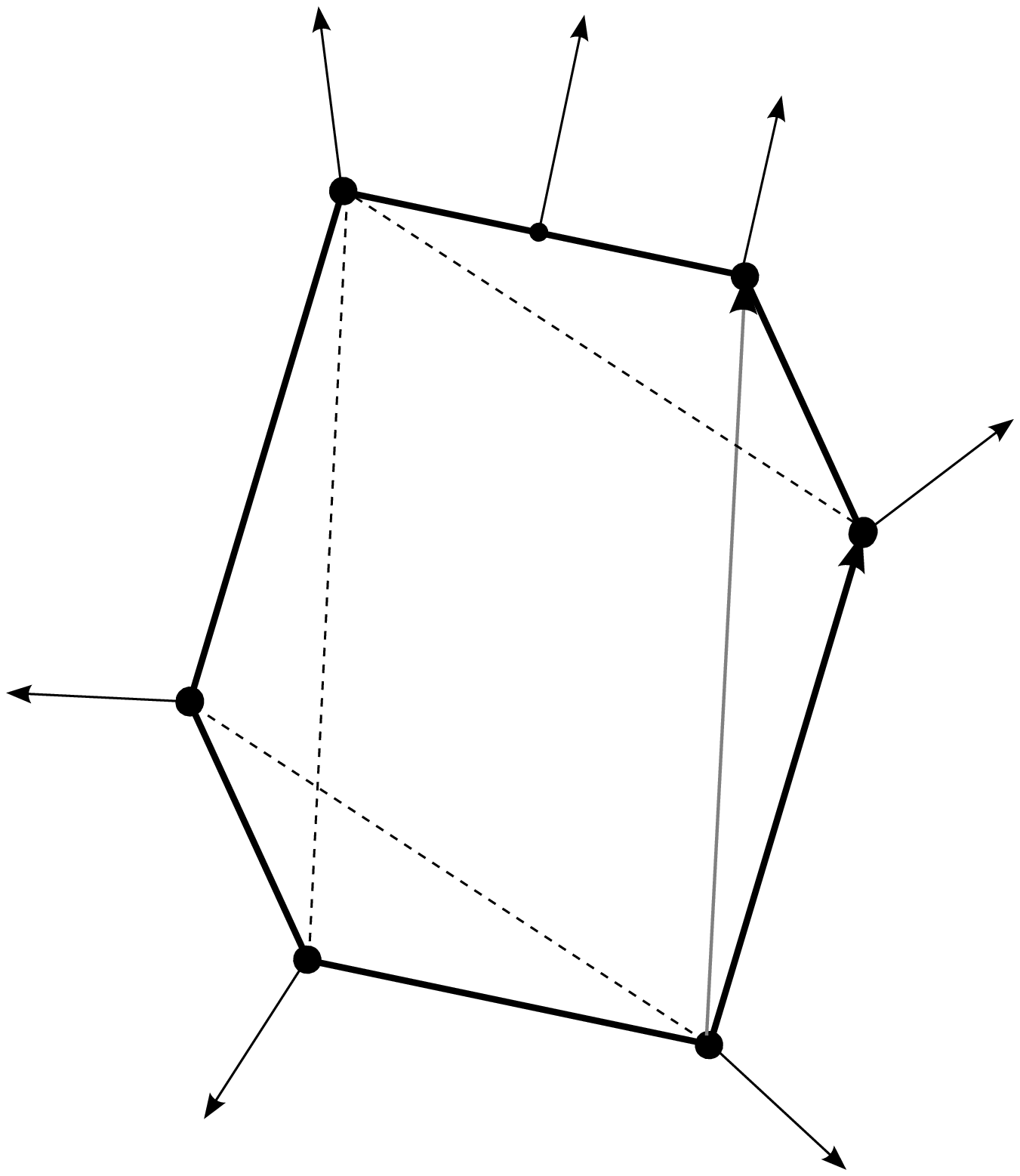}}
	\put(-3,49){\scriptsize $p_1(x)=p_1(y)$}
	\put(22,53){\scriptsize $h$}
	\put(33,10){\scriptsize $p_3(x)=p_3(y)$}
	\put(36,20){\scriptsize $y$}
	\put(31,30){\scriptsize $x$}
	\put(6,13.5){\scriptsize $p_2(x)$}
	\put(2,23){\scriptsize $p_2(y)$}

        \end{picture}
	&
	\begin{picture}(95,60)
	\put(1,7){\IncludeGraph{width=90\unitlength}{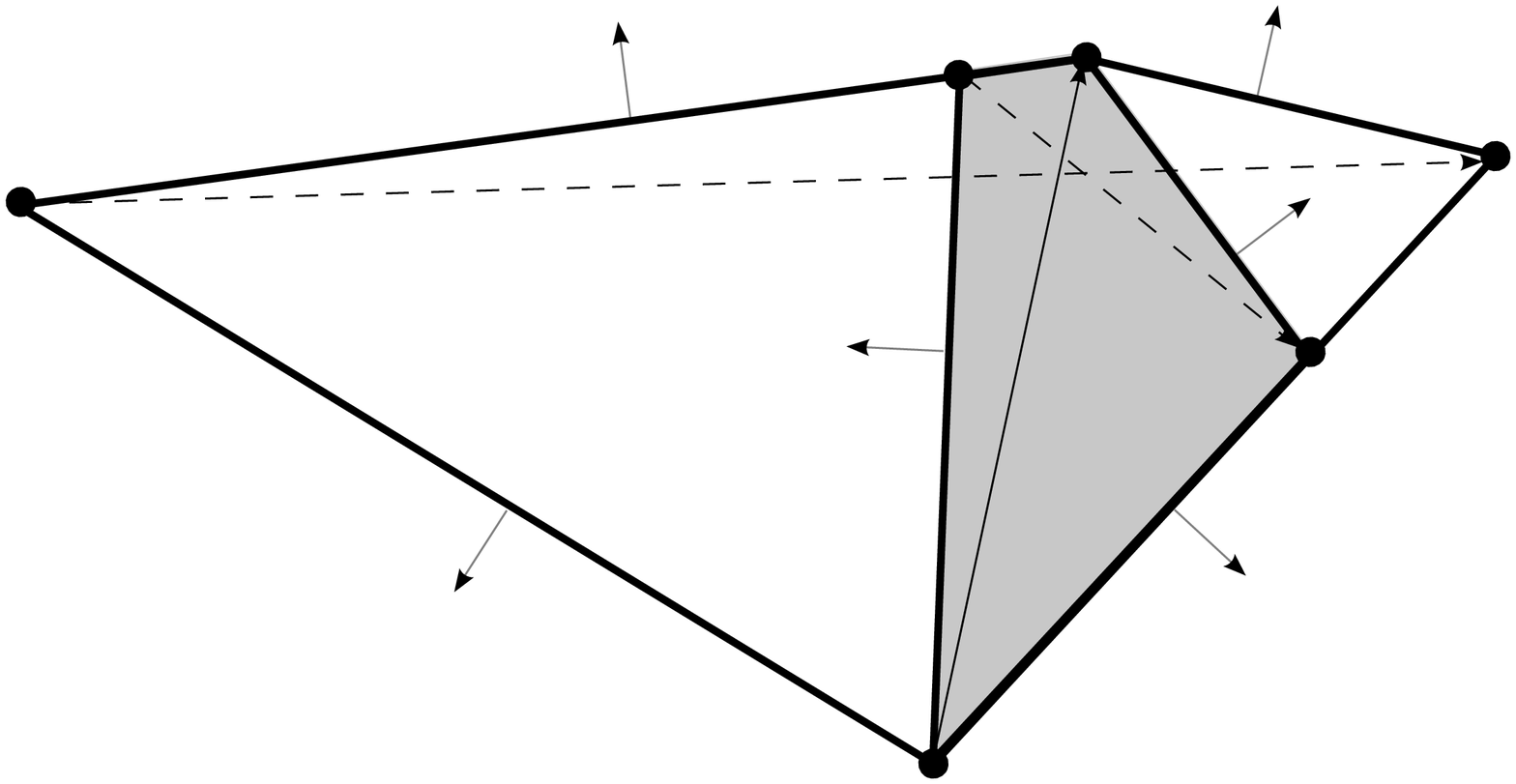}}
	\put(63.5,32){\scriptsize $h$}
	\put(57,5){\scriptsize $q_3$}
	\put(66,52){\scriptsize $q_1$}
	\put(-2,43){\scriptsize $q_2(x,h)$}
	\put(92,42){\scriptsize $q_4(x,h)$}	
	\put(51,50.5){\scriptsize $q_2(y,h)$}
	\put(81,30.5){\scriptsize $q_4(y,h)$}
	\put(25,53){\scriptsize $u_1(x)=u_1(y)$}
	\put(26,16){\scriptsize $u_2(x)$}
	\put(46,32){\scriptsize $u_2(y)$}
	\put(79,53){\scriptsize $u_4(x)$}
	\put(77,18){\scriptsize $u_3(x)=u_3(y)$}

        \end{picture}
	\\
	\parbox[t]{55\unitlength}{\mycaption{The hexagon $H$ and the normals of $K$ at the vertices of $H$ \label{HexNormDetermFig}}}
	&
	\parbox[t]{95\unitlength}{\mycaption{The boundaries of $Q(x,h)$ and $Q(y,h)$ are plotted in bold; $Q(y,h)$ is shaded.\label{HexNormDetermFig2}}}
	\end{FigTab}

Let $h$ be an outward normal of the side $[p_1(x),p_4(x)]$ of $H.$ Let $Q(x,h)$ and $Q(y,h)$ be quadrilaterals constructed as in the statement of Lemma~\ref{HessQuadInterp}. By the choice of $H$ we have $u_1(x)=u_1(y),$ $u_3(x)=u_3(y),$ while $u_2(x)$ follows $u_2(y),$ and $u_4(x)$ follows $u_4(y),$ in counterclockwise order on $\usphere^1.$  Consequently, $[o,h]$ is a common diagonal of $Q(x,h)$ and $Q(y,h),$ while the vertices $q_2(y,h)$ and $q_4(y,h)$ of $Q(y,h)$ lie in the relative interiors of the sides $[q_1(x,h),q_2(x,h)]$ and $[q_3(x,h),q_4(x,h)],$ respectively, of $Q(x,h);$ see Fig.~\ref{HexNormDetermFig2}. The latter implies that the diagonals $[q_2(x,h),q_4(x,h)]$ and $[q_2(y,h),q_4(y,h)]$ of the quadrilaterals $Q(x,h)$ and $Q(y,h),$ respectively, are not parallel. Hence, by Lemma~\ref{HessQuadInterp}, $\G(x) h$ and $\G(y) h$ are not parallel, which implies that the chosen $h$ is not an eigenvector of $\G(x) \G(y)^{-1}.$  Consequently, $\lin u_1(x)$ and $\lin u_3(x)$ are two distinct eigenspaces of $\G(x) \G(y)^{-1},$ and we arrive at the assertion.
\end{proof}

\begin{proposition} \label{HexDetermProp}
	Let $K$ be a strictly convex and $\cclass^1$ regular body in $\E^2.$ Let $H$ be a centrally symmetric convex hexagon with consecutive vertices $h_1, \ldots, h_6$ in counterclockwise order. For $i \in \{1,2,3\}$ we introduce the vectors $x_i:=h_{2i +1} - h_{2 i -1}$;
	 see Fig.~\ref{HexDetermNewFig}.  Then a translate of $H$ is inscribed in $K$ if and only if
	\begin{equation}
		D(x_i) = h_{2i+2}-h_{2i+1}
\label{DHexEq}
	\end{equation}
	for every $i \in \{1,2,3\}$ and 
	\begin{equation}
		\prod_{i=1}^3 \bigl(1+\hessg(x_i) \bigr) \ge 0. \label{1+Delta:Prod}
	\end{equation}
\end{proposition}
	\begin{FigTab}{c}{1mm}
	        \begin{picture}(48,42)
	\put(3,2){\IncludeGraph{width=40\unitlength}{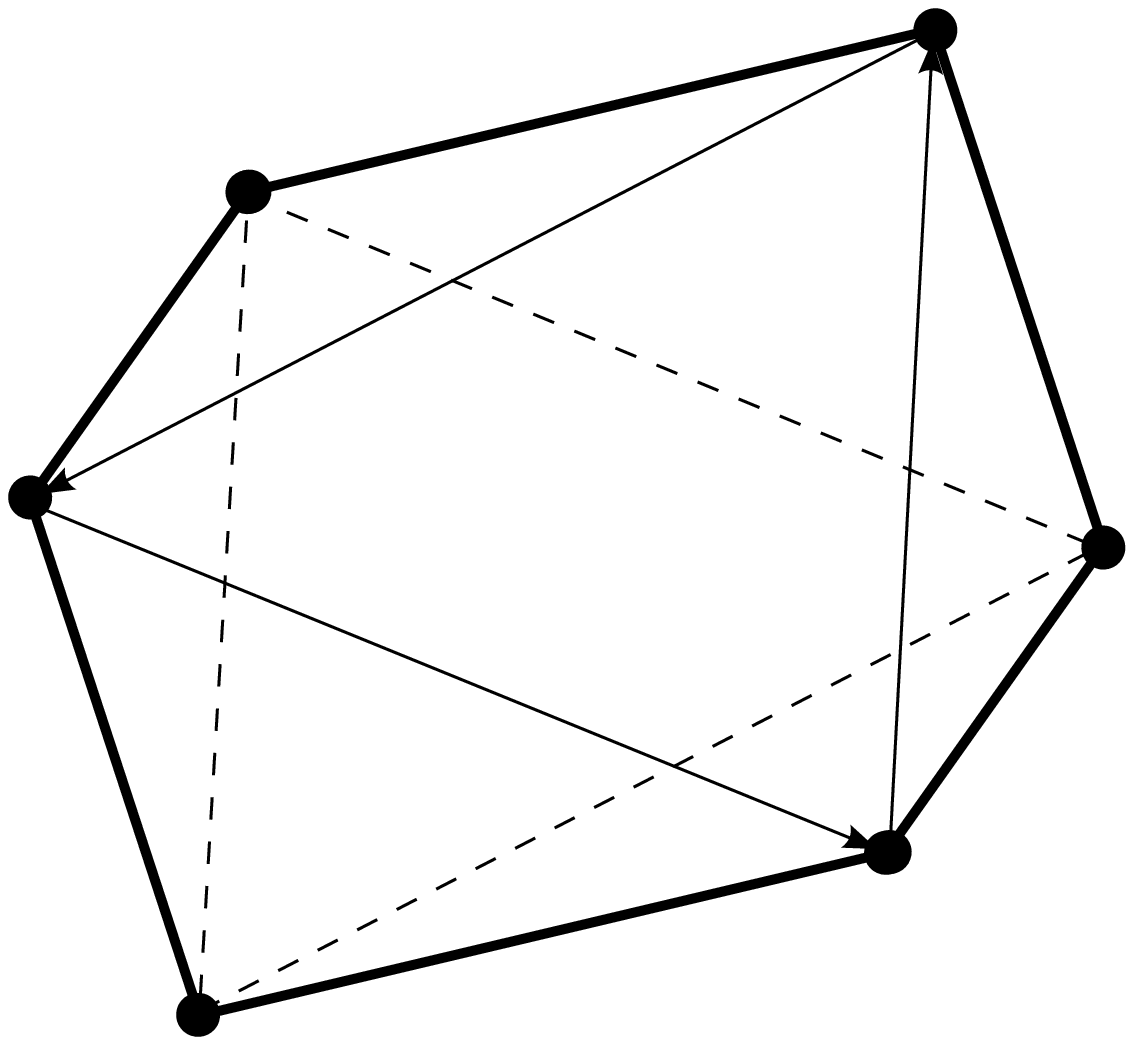}}
	\put(39,39){\scriptsize $h_1$}
	\put(9.5,33.5){\scriptsize $h_2$}
	\put(0.5,21){\scriptsize $h_3$}
	\put(8,0){\scriptsize $h_4$}
	\put(36.5,5.5){\scriptsize $h_5$}
	\put(45,18){\scriptsize $h_6$}	
	\put(25,30){\scriptsize $x_1$}
	\put(21,12){\scriptsize $x_2$}
	\put(33.5,27){\scriptsize $x_3$}

        \end{picture}
	\\
	\parbox[t]{0.90\textwidth}{\mycaption{\label{HexDetermNewFig} The hexagon $H$ and the vectors $x_1$, $x_2$ and $x_3$.}}
	\end{FigTab}
\begin{proof}	
	Let us show the necessity. Since conditions \eqref{DHexEq} and \eqref{1+Delta:Prod} are invariant with respect to translations  of $H,$ we can assume that $H$ itself is inscribed in $K.$ From the definition of $x_i$ and $p_j$ it follows that
	\begin{equation} \label{hrep}
	\begin{array}{ccc}
		p_1(x_i) = h_{2i+2}, & \qquad & p_2(x_i)=h_{2i-2}, \\
		p_3(x_i) = h_{2i-1}, & \qquad & p_4(x_i) = h_{2i+1},
	\end{array}
	\end{equation}
	where $i \in \integer.$ Thus, \eqref{DHexEq} follows directly form \eqref{hrep} and the definition of the function $D.$ Let us obtain \eqref{1+Delta:Prod}. By \eqref{OnePlusDelta} we have 
	\begin{equation*}
		s:=\prod_{i=1}^3 \bigl(1 + \hessg(x_i) \bigr) = \frac{s_1 \, s_2}{s_3} ,
	\end{equation*}
	where
	\begin{align*}
		s_1 &:= \prod_{i=1}^3 \detuvar{1}{3}{x_i}, \\
		s_2 &:= \prod_{i=1}^3 \detuvar{2}{4}{x_i}, \\
		s_3 &:= \prod_{i=1}^3 \detuvar{1}{2}{x_i} \, \detuvar{3}{4}{x_i}. 
	\end{align*}
	The determinants
	$\detuvar{1}{2}{x_i}$ and $\detuvar{3}{4}{x_i}$ are strictly positive; see \eqref{07.10.24,13:31}. Consequently $s_3 > 0.$
From \eqref{hrep} we get the equalities $p_1(x_{i+1})=p_2(x_i)$ and $p_3(x_{i+1})=p_4(x_i)$ and by this also the equalities
$$\detuvar{1}{3}{x_{i+1}}=\detuvar{2}{4}{x_i}$$ for $i \in \{1,2,3\}.$ Hence we see that $s_1=s_2$ and therefore $s \ge 0.$

	Now let us show the sufficiency by contradiction. Assume that for some $K$ and $H,$ satisfying the assumptions of the proposition, conditions \eqref{DHexEq} and \eqref{1+Delta:Prod} are fulfilled but no translate of $H$ is inscribed in $K.$ By \eqref{DHexEq} we see that for every $i \in \integer$ the parallelogram $P(x_i)$ is a translate of $\conv \{h_{2i-2},h_{2i-1},h_{2i+1},h_{2i+2} \}$ and, moreover, one has
	\begin{equation} \label{hrep:trans}
	\begin{array}{ccc}
		p_1(x_i) = a_i+ h_{2i+2}, & \qquad & p_2(x_i)= a_i + h_{2i-2}, \\
		p_3(x_i) = a_i + h_{2i-1}, & \qquad & p_4(x_i) = a_i + h_{2i+1},
	\end{array}
	\end{equation}
	with appropriate $a_i \in \E^2.$ If for some $i, j \in \{1,2,3\}$ with $i \ne j$ the parallelograms $P(x_i),$ $P(x_j)$ share a diagonal, it follows that $H$ is a translate of  $\conv \bigl(P(x_i) \cup P(x_j) \bigr),$ a contradiction. Now we consider the case when no two distinct parallelograms $P(x_i)$ and $P(x_j)$ share a diagonal. Let $i \in \integer.$ In view of \eqref{hrep:trans}, we get that 
	$$h_{2i-2}-h_{2i+1}=p_1(x_{i+1})-p_3(x_{i+1})=p_2(x_i)-p_4(x_i).$$
	Thus the diagonals  $[p_1(x_{i+1}),p_3(x_{i+1})]$ and $[p_2(x_i),p_4(x_i)]$ of $P(x_{i+1})$ and $P(x_i),$ respectively, are translates of each other. By the assumption, these diagonals are distinct. Thus, $[p_1(x_{i+1}),p_3(x_{i+1})]$ and $[p_2(x_i),p_4(x_i)]$ are distinct chords of $K$ which are translates of $[h_{2i+1},h_{2i-2}];$ see Fig.~\ref{ThreeHexagons}. 

	\begin{FigTab}{c}{1.2mm}
	        \begin{picture}(48,44)
	\put(3,4){\IncludeGraph{width=40\unitlength}{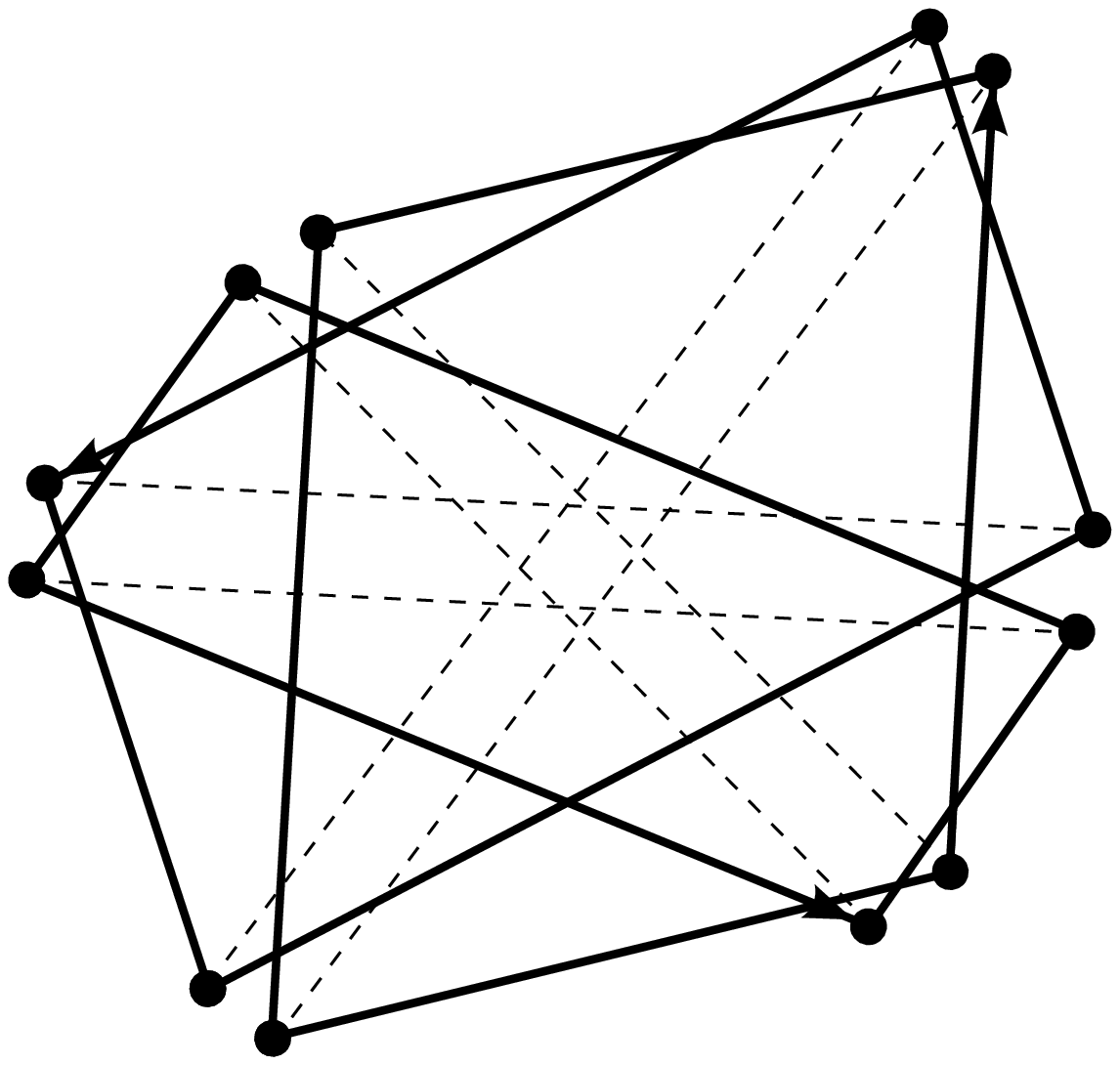}}
	\put(32,43){\scriptsize $p_3(x_1)$}
	\put(-3,26){\scriptsize $p_4(x_1)$}
	\put(2,5){\scriptsize $p_1(x_1)$}
	\put(45.5,23){\scriptsize $p_2(x_1)$}
	\put(-4.3,20){\scriptsize $p_3(x_2)$}
	\put(30,5.8){\scriptsize $p_4(x_2)$}
	\put(44,18){\scriptsize $p_1(x_2)$}
	\put(3.7,32){\scriptsize $p_2(x_2)$}
	\put(40,10){\scriptsize $p_3(x_3)$}	
	\put(41.8,40){\scriptsize $p_4(x_3)$}	
	\put(12,36){\scriptsize $p_1(x_3)$}	
	\put(12,1.5){\scriptsize $p_2(x_3)$}	

        \end{picture}
	\\
	\parbox[t]{0.90\textwidth}{\mycaption{The parallelograms $P(x_1), P(x_2), P(x_3)$ and  their diagonals.\label{ThreeHexagons}}}
	\end{FigTab}

The strict convexity of $K$ implies that  $$\sign \detuvar{1}{3}{x_{i+1}} = - \sign \detuvar{2}{4}{x_i} \ne 0$$ for $1 \le i \le 3.$ The latter yields $\sign s_1 = - \sign s_2 \ne 0.$ But since $s_3>0$ we obtain that $s<0,$ a contradiction to \eqref{1+Delta:Prod}.
\end{proof}

\begin{proof}[Proof of Proposition~\ref{u1u3:determ}]
	First we show that there exists $y \in \IDKO$ with $y \ne x$ such that $p_i(x)=p_i(y)$ for $i \in \{1,3\}.$  Let $c$ be the center of $P(x)$ and ${K_c}$ be the reflection of $K$ with respect to $c;$ see Fig.~\ref{HexNoAffDiag}.   Assume first that $1 + \hessg(x) >0.$ Then, in view of \eqref{OnePlusDelta} and \eqref{07.10.24,13:31}, $$\sign \detuvar{1}{3}{x} = \sign \detuvar{2}{4}{x} \ne 0.$$ 
	Therefore $\bd K$ and $\bd {K_c}$ intersect transversally at $p_i(x)$ for every $i \in \{1,\ldots,4\}.$ Moreover, either a small subarc of $[p_1(x),p_2(x)]_K$  with endpoint $p_1(x)$ is contained in ${K_c}$ and a small subarc of $[p_1(x),p_2(x)]_K$ with endpoint $p_2(x)$ is contained in $\E^2 \setminus \intr {K_c}$ or vice versa (that is, a small subarc of $[p_1(x),p_2(x)]_K$  with endpoint $p_1(x)$ is contained in $\E^2 \setminus \intr {K_c}$ and a small subarc of $[p_1(x),p_2(x)]_K$ with endpoint $p_2(x)$ is contained in ${K_c}$). Consequently, the arcs $[p_1(x),p_2(x)]_K$ and $[p_1(x),p_2(x)]_{K_c}$ intersect at some point $q$ distinct from $p_1(x)$ and $p_2(x).$ We define $y := p_1(x)-q.$ By construction,  $p_1(x), q, p_3(x),$ and $2 c - q$ are consecutive vertices of $P(y)$; see Fig.~\ref{HexNoAffDiag}. Therefore $y$ satisfies the desired conditions. 

    \begin{FigTab}{c}{2mm}
     \begin{picture}(38,24)
        \put(5,2){\ExternalFigure{0.31}{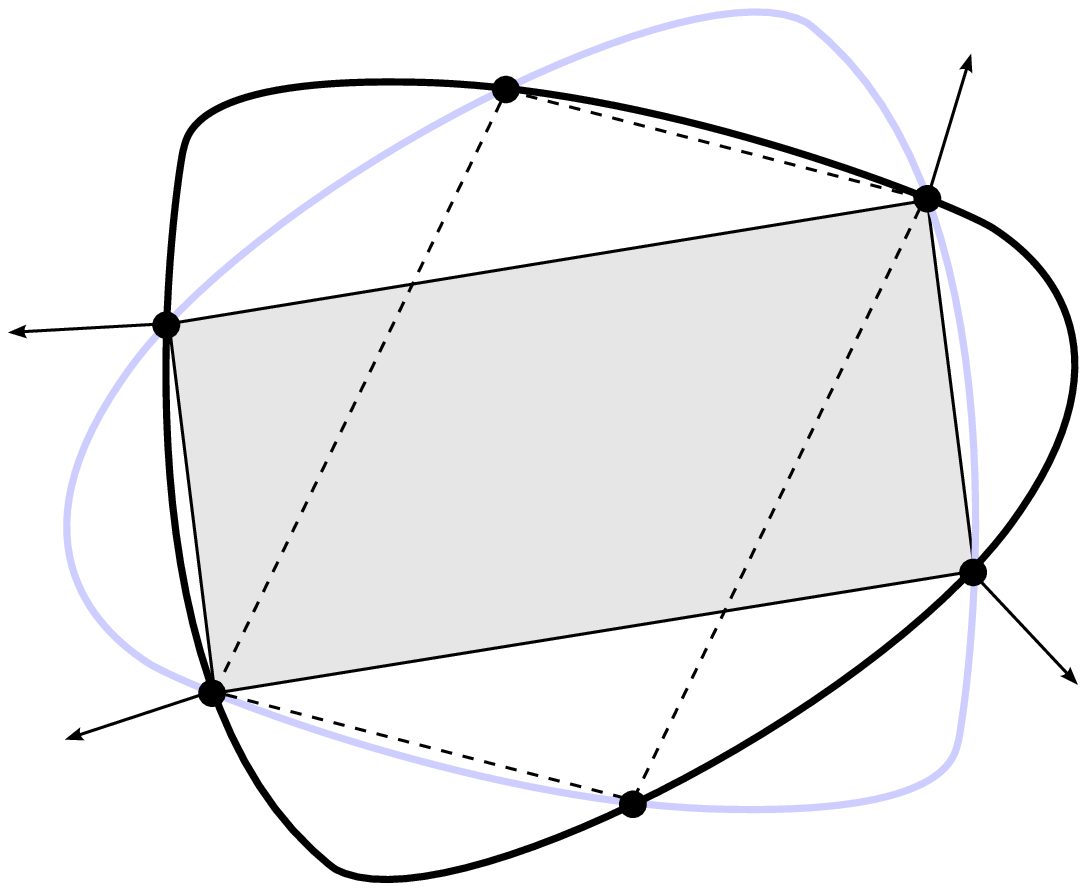}}
	\put(23.5,16.5){\scriptsize $p_1(x)$}
	\put(10,14){\scriptsize $p_2(x)$}
	\put(10.3,7.7){\scriptsize $p_3(x)$}
	\put(24.5,9.5){\scriptsize $p_4(x)$}
	\put(16.5,21.5){\scriptsize $q$}	
	\put(27,22){\scriptsize $u_1(x)$}
	\put(1.5,15){\scriptsize $u_2(x)$}
	\put(3,5){\scriptsize $u_3(x)$}
	\put(30.5,6){\scriptsize $u_4(x)$}
	\put(20,2){\scriptsize $2c-q$}

        \end{picture} \\
	\parbox[t]{0.90\textwidth}{\mycaption{The bodies $K$ and $K_c$ and the parallelograms $P(x)$ and $P(y).$\label{HexNoAffDiag}}}
    \end{FigTab}

	In the case $1 + \det \G(x) <0$ we can use similar arguments showing that the arcs $[p_4(x),p_1(x)]_K$ and $[p_4(x),p_1(x)]_{K_c}$ intersect at some point $q$ distinct from $p_2(x)$ and $p_3(x).$ Thus, for that case we can define $y:= q - p_3(x).$ 

	Now let $L$ be a strictly convex and $\cclass^1$ regular planar convex body with the same covariogram as $K.$ By Proposition~\ref{HexDetermProp}, a translate of $H:=\conv \bigl(P(K,x) \cup P(K,y)\bigr)$ is inscribed in $L.$  Without loss of generality we assume $H$ itself is inscribed in $L,$ that is, $P(K,x)=P(L,x)$ and $P(K,y)=P(L,y).$ Notice that the inequality $1 + \det \G(x) \ne 0$ implies  that $[p_1,p_3]$ is not an affine diameter of $K$ or $L.$ Then, by Lemma~\ref{u1u3eigv}, we have $\{u_1(K,x),-u_3(K,x)\}=\{u_1(L,x),-u_3(L,x)\},$ and we are done.
\end{proof}

\begin{proof}[Proof of Proposition~\ref{ArcDetermProp}] 
	By Theorem~\ref{central symmetry}, if $K$ is centrally symmetric, then so is $L.$ In this case $K$ and $L$ are translates of $\frac{1}{2} \supp g_K = \frac{1}{2} \supp g_L,$ and the proof is concluded. 

	Now assume that $K$ is not centrally symmetric. Then, by Theorem~\ref{central symmetry}, there exists $x_0 \in \IDKO,$ such that $\det G(x_0) \ne -1.$ This implies $u_1(K,x_0) \ne -u_3(K,x_0).$ Let $N_1$ and $N_3$ be disjoint open neighborhoods of $u_1(K,x_0)$ and $-u_3(K,x_0),$ respectively. In view of Theorem~\ref{frst:drv:cov:thm} (Part~\ref{Ptrans:prop}) and Proposition~\ref{u1u3:determ}, replacing $L$ by an appropriate translation or reflection, we can assume that $P(K,x_0)=P(L,x_0)$ and $u_i(K,x_0)=u_i(L,x_0)$ for $i \in \{1,3\}.$ Let $q(t), \ 0 \le t \le 1,$ be a continuous, counterclockwise parametrization of a small boundary arc of $K$ such that $q(0)=p_4(K,x_0)$ and, for $x(t):=q(t)-p_3(K,x_0)$, one has $\det G\bigl(x(t)\bigr) \ne -1$, $u_1\bigl(K,x(t)\bigr) \in N_1$ and $-u_3\bigl(K,x(t)\bigr) \in N_3$ for every $0 \le t \le 1$.  

We show by contradiction that for every $0 \le t \le 1$ the equalities 
\begin{equation} \label{07.11.19,09:41}
	u_i\bigl(K,x(t)\bigr)=u_i\bigl(L,x(t)\bigr), \ \ i \in \{1,3\},
\end{equation} are fulfilled. Assume the contrary. Then, by Proposition~\ref{u1u3:determ}, there exists $t_1$ with $0 < t_1 \le 1$ such that $u_1\bigl(K,x(t_1)\bigr) = -u_3\bigl(L,x(t_1)\bigr)$ and $-u_3\bigl(K,x(t_1)\bigr)=u_1\bigl(L,x(t_1)\bigr).$ 
In particular we have $u_1\bigl(L,x(t_1)\bigr)\in N_3$. Since $u_1\bigl(L,x(0)\bigr) = u_1\bigl(K,x(0)) \in N_1$ and since $N_1$ and $N_3$ are disjoint, 
%But in view of the equalities $u_1\bigl(K,x(0))=u_1\bigl(L,x(0)\bigr),$ and the choice of $N_1$ and $N_3$ 
there exists $t_2$ with $0 < t_2 < t_1$ such that $u_1\bigl(L,x(t_2)\bigr)$ lies outside $N:=N_1 \cup N_3.$ Hence $  \{u_1\bigl(L,x(t_2)\bigr),-u_3\bigl(L,x(t_2)\bigr)\} \not\subseteq N.$ But, by construction, we have  $\{u_1\bigl(K,x(t_2)\bigr), -u_3\bigl(K,x(t_2)\bigr)\} \subseteq N,$ a contradiction to Proposition~\ref{u1u3:determ}. 

The definition of $x(t)$ implies  $p_3(K,x(t))=p_3(K,x_0)$, for each $t \in [0,1]$, and therefore, it also implies $u_3(K,x(t))=u_3(K,x_0)$. Hence, in view of \eqref{07.11.19,09:41}, we get $u_3(L,x(t))=u_3(L,x_0).$ Consequently, by the strict convexity of $L$, we also have $p_3(L,x(t))=p_3(L,x_0)$.
%Thus, by the choice of $x(t),$ we have $u_3(x_0)=u_3\bigl(K,x(t)\bigr)=u_3\bigl(L,x(t)\bigr)$ for $1 \le t \le 1.$ Consequently $p_3(x_0)=p_3\bigl(K,x(t)\bigr)=p_3\bigl(L,x(t)\bigr)$ and, moreover, $P\bigl(K,x(t)\bigr)= P\bigl(L,x(t)\bigr)$ for $0 \le t \le 1.$ 
The latter implies 
	$$
		\bigl[p_4\bigl(K,x(0)\bigr),p_4\bigl(K,x(1)\bigr)\bigr]_K = \bigl[p_4\bigl(L,x(0)\bigr),p_4\bigl(L,x(1)\bigr)\bigr]_L,
	$$
	and concludes the proof.
\end{proof}

\small

\newcommand{\etalchar}[1]{$^{#1}$}
\def\cprime{$'$} \def\cprime{$'$} \def\cprime{$'$} \def\cprime{$'$}
\providecommand{\bysame}{\leavevmode\hbox to3em{\hrulefill}\thinspace}
\providecommand{\MR}{\relax\ifhmode\unskip\space\fi MR }
\providecommand{\MRhref}[2]{%
  \href{http://www.ams.org/mathscinet-getitem?mr=#1}{#2}
}
\providecommand{\href}[2]{#2}

\begin{tabular}{lr}
\begin{tabular}{l}
  Gennadiy Averkov, \\
  Faculty of Mathematics, \\
  University of Magdeburg, \\ Universit\"atsplatz 2, \\
  D-39106 Magdeburg, \\
  Germany \\
  \emph{e-mail:} gennadiy.averkov@googlemail.com
\end{tabular}
&
\begin{tabular}{l}
  Gabriele Bianchi, \\
  Department of Mathematics, \\
  Universit\`{a} di Firenze, \\ Viale Morgagni  67a, \\	
  50134 Firenze, \\
  Italy \\
  \emph{e-mail:} gabriele.bianchi@unifi.it
\end{tabular}
\end{tabular}
\end{document}